\documentclass[10pt,a4paper]{amsart}
\usepackage{amsmath,amsfonts,amsthm,amssymb,graphicx,mathtools}
\usepackage[stretch=10]{microtype}
\usepackage[unicode]{hyperref}
\usepackage{xcolor}
\definecolor{dark-red}{rgb}{0.4,0.15,0.15}
\definecolor{dark-blue}{rgb}{0.15,0.15,0.4}
\definecolor{medium-blue}{rgb}{0,0,0.5}
\hypersetup{
    colorlinks, linkcolor={dark-red},
    citecolor={dark-blue}, urlcolor={medium-blue}
}
\makeatletter
\newcommand*{\defeq}{\mathrel{\rlap{%
                     \raisebox{0.3ex}{$\m@th\cdot$}}%
                     \raisebox{-0.3ex}{$\m@th\cdot$}}%
                     =}
\newcommand*{\eqdef}{\mathrel{=\llap{%
                     \raisebox{0.3ex}{$\m@th\cdot$}}%
                     \llap{\raisebox{-0.3ex}{$\m@th\cdot$}}}%
                     }
\makeatother
\newcommand\A{\mathbb{A}}
\renewcommand\aa{\mathfrak{a}}
\newcommand\bb{\mathfrak{b}}
\renewcommand\C{\mathbb{C}}
\newcommand\cc{\mathfrak{c}}
\newcommand\FF{\mathcal{F}}
\newcommand\GL{\mathrm{GL}}
\newcommand\Hb{\mathbb{H}}
\newcommand\II{\mathcal{I}}
\newcommand\Q{\mathbb{Q}}
\newcommand\R{\mathbb{R}}
\newcommand\SL{\mathrm{SL}}
\newcommand\Z{\mathbb{Z}}
\def\e{\varepsilon}
\DeclareMathOperator{\sgn}{sgn}
\numberwithin{equation}{section}
\newtheorem{theorem}[equation]{Theorem}
\newtheorem{corollary}[equation]{Corollary}
\newtheorem{lemma}[equation]{Lemma}
\newtheorem{proposition}[equation]{Proposition}
\theoremstyle{remark}
\newtheorem{remark}[equation]{Remark}
\begin{document}

\title{Effective Lower Bounds for \texorpdfstring{$L(1,\chi)$}{L(1,\83\307)} via Eisenstein Series}

\author{Peter Humphries}


\address{Department of Mathematics, Princeton University, Princeton, New Jersey 08544, USA}

\email{\href{mailto:peterch@math.princeton.edu}{peterch@math.princeton.edu}}

\keywords{Dirichlet $L$-function, lower bounds, Eisenstein series}

\subjclass[2010]{11M20 (primary); 11M36 (secondary).}

\begin{abstract}
We give effective lower bounds for $L(1,\chi)$ via Eisenstein series on $\Gamma_0(q) \backslash \Hb$. The proof uses the Maa\ss{}--Selberg relation for truncated Eisenstein series and sieve theory in the form of the Brun--Titchmarsh inequality. The method follows closely the work of Sarnak in using Eisenstein series to find effective lower bounds for $\zeta(1 + it)$.
\end{abstract}

\maketitle


\section{Introduction}

Let $q$ be a positive integer, let $\chi$ be a Dirichlet character modulo $q$, and let
\[L(s,\chi) \defeq \sum_{n = 1}^{\infty} \frac{\chi(n)}{n^s}\]
be the associated Dirichlet $L$-function, which converges absolutely for $\Re(s) > 1$ and extends holomorphically to the entire complex plane except when $\chi$ is principal, in which case there is a simple pole at $s = 1$. It is well known that Dirichlet's theorem on the infinitude of primes in arithmetic progressions is equivalent to showing that $L(1,\chi) \neq 0$ for every Dirichlet character $\chi$ modulo $q$. Of further interest is obtaining lower bounds for $L(1,\chi)$ in terms of $q$. By complex analytic means \cite[Theorems 11.4 and 11.11]{Montgomery}, one can show that if $\chi$ is complex, then
\[|L(1,\chi)| \gg \frac{1}{\log q},\]
while
\[L(1,\chi) \gg \frac{1}{\sqrt{q}}\]
if $\chi$ is quadratic. In both cases, the implicit constants are effective. For quadratic characters, the Landau--Siegel theorem states that
\[L(1,\chi) \gg_{\varepsilon} q^{-\varepsilon}\]
for all $\varepsilon > 0$ \cite[Theorem 11.14]{Montgomery}, though this estimate is ineffective due to the possible existence of a Landau--Siegel zero of $L(s,\chi)$.

In this article, we give a novel proof of effective lower bounds for $L(1,\chi)$, albeit in slightly weaker forms.

\begin{theorem}\label{mainthm}
Let $q \geq 2$ be a positive integer, and let $\chi$ be a primitive character modulo $q$. If $\chi$ is complex, then
\[|L(1,\chi)| \gg \frac{1}{(\log q)^3},\]
while
\[L(1,\chi) \gg \frac{1}{\sqrt{q} (\log q)^2}\]
if $\chi$ is quadratic. In both cases, the implicit constants are effective.
\end{theorem}

Our proof of \hyperref[mainthm]{Theorem \ref*{mainthm}} makes use the fact that $L(s,\chi)$ appears in the Fourier expansion of an Eisenstein series associated to $\chi$ on $\Gamma_0(q) \backslash \Hb$, together with sieve theory --- specifically the Brun--Titchmarsh inequality --- to find these lower bounds. As is well-known, improving the constant in the Brun--Titchmarsh inequality is essentially equivalent the nonexistence of Landau--Siegel zeroes; it is for this same reason that the lower bounds in \hyperref[mainthm]{Theorem \ref*{mainthm}} are weak for quadratic characters, as we discuss in \hyperref[LandauSiegelremark]{Remark \ref*{LandauSiegelremark}}.

That one can use Eisenstein series to prove nonvanishing of $L$-functions is well-known, first appearing in unpublished work of Selberg, but such methods were not shown to give good effective lower bounds for $L$-functions on the line $\Re(s) = 1$ until the work of Sarnak \cite{Sarnak}. He showed that
\[|\zeta(1 + it)| \gg \frac{1}{(\log|t|)^3}\]
for $|t| > 1$ by exploiting the inhomogeneous form of the Maa\ss{}--Selberg relation for the Eisenstein series $E(z,s)$ for the group $\SL_2(\Z)$.

More precisely, for $t > 1$, Sarnak studied the integral
\[\II \defeq \int_{1/t}^{\infty} \int_{0}^{1} \left|\zeta(1 + 2it)\right|^2 \left|\Lambda^{t}\left(z,\frac{1}{2} + it\right)\right|^2 \, \frac{dx \, dy}{y^2}\]
involving a truncated Eisenstein series $\Lambda^T E(z,s)$ and found an upper bound up to a scalar multiple for this integral of the form
\[t (\log t)^2 |\zeta(1 + 2it)|\]
via the Maa\ss{}--Selberg relation, and a lower bound up to a scalar multiple of the form
\[\frac{1}{t} \sum_{\frac{t^2}{8} \leq m \leq \frac{t^2}{4}} \left|\sigma_{-2it}(m)\right|^2\]
via Parseval's identity, where
\[\sigma_{-2it}(m) \defeq \sum_{d \mid m} d^{-2it}.\]
By restricting the summation over $m$ to primes, Sarnak was able to use sieve theory to show that
\[\sum_{\frac{t^2}{8} \leq p \leq \frac{t^2}{4}} \left|\sigma_{-2it}(p)\right|^2 \gg \frac{t^2}{\log t},\]
from which the result follows. Indeed, the use of sieve theory to prove lower bounds for $\zeta(1 + it)$ (and also $L(1 + it,\chi)$) has its roots in work of Balasubramanian and Ramachandra \cite{Balasubramanian}.

The chief novelty of Sarnak's work is to use the Maa\ss{}--Selberg relation to obtain effective lower bounds for $\zeta(1 + it)$; more precisely, it is the inhomogeneous nature of the Fourier expansion of the Eisenstein series $E(z,s)$, whose constant term involves $\zeta(2s - 1)/\zeta(2s)$ and whose nonconstant terms involve $1/\zeta(2s)$. This method has been generalised by Gelbart and Lapid \cite{Gelbart} to determine effective lower bounds on the line $\Re(s) = 1$ for $L$-functions associated to automorphic representations on arbitrary reductive groups over number fields, albeit with the lower bound being in the weaker form $C|t|^{-n}$ for some constants $C,n$ depending on the $L$-function, for Gelbart and Lapid make no use of sieve theory in this generalised setting. More recently, Goldfeld and Li \cite{Goldfeld} have succeeded in generalising Sarnak's method to show that
\[\left|L\left(1 + it,\pi \times \widetilde{\pi}\right)\right| \gg_{\pi} \frac{1}{(\log|t|)^3}\]
for any cuspidal automorphic representation $\pi$ of $\GL_n(\A_{\Q})$ that is unramified and tempered at every place, with the implicit constant in the lower bound dependent on $\pi$.

All three of these results give lower bounds for $L$-functions on the line $\Re(s) = 1$ in the height aspect, namely in terms of $t$. In this article, we give the first example of Sarnak's method being used to give lower bounds for $L$-functions on the line $\Re(s) = 1$ in the level aspect, namely in terms of $q$.

\section{Eisenstein Series}

We introduce Eisenstein series for the group $\Gamma_0(q)$ associated to a primitive Dirichlet character $\chi$ modulo $q$. Standard references for this material are \cite{Deshouillers}, \cite{Duke}, and \cite{Iwaniec}.

\subsection{Cusps}

Let $\Hb$ be the upper half plane, upon which $\SL_2(\R)$ acts via M\"{o}bius transformations $\gamma z = \frac{az + b}{cz + d}$ for $\gamma = \left(\begin{smallmatrix} a & b \\ c & d \end{smallmatrix}\right) \in \SL_2(\R)$ and $z \in \Hb$. Let $q$ be a positive integer, and let $\aa$ be a cusp of $\Gamma_0(q) \backslash \Hb$, where
\[\Gamma_0(q) \defeq \left\{\begin{pmatrix} a & b \\ c & d \end{pmatrix} \in \SL_2(\Z) \colon c \equiv 0 \hspace{-.2cm} \pmod{q}\right\},\]
and we denote the stabiliser of $\aa$ by
\[\Gamma_{\aa} \defeq \left\{\gamma \in \Gamma_0(q) \colon \gamma \aa = \aa\right\}.\]
This subgroup of $\Gamma_0(q)$ is generated by two parabolic elements $\pm \gamma_{\aa}$, where
\[\gamma_{\aa} \defeq \sigma_{\aa} \begin{pmatrix} 1 & 1 \\ 0 & 1 \end{pmatrix} \sigma_{\aa}^{-1},\]
and the scaling matrix $\sigma_{\aa} \in \SL_2(\R)$ is such that
\[\sigma_{\aa} \infty = \aa, \qquad \sigma_{\aa}^{-1} \Gamma_{\infty} \sigma_{\aa} = \Gamma_{\infty},\]
where
\[\Gamma_{\infty} \defeq \left\{\pm \begin{pmatrix} 1 & n \\ 0 & 1 \end{pmatrix} \in \Gamma_0(q) \colon n \in \Z\right\}\]
is the stabiliser of the cusp at infinity. The scaling matrix is unique up to translation on the right.

Let $\chi$ be a primitive character modulo $q$. A cusp $\aa$ of $\Gamma_0(q) \backslash \Hb$ is said to be singular with respect to $\chi$ if $\chi(\gamma_{\aa}) = 1$, where $\chi(\gamma) \defeq \chi(d)$ for $\gamma = \left(\begin{smallmatrix} a & b \\ c & d \end{smallmatrix}\right) \in \Gamma_0(q)$. As $\chi$ is primitive, any singular cusp is equivalent to $1/v$ for a single unique divisor $v$ of $q$ satisfying $vw = q$ and $(v,w) = 1$, where $w$ is the width of the cusp; when $v = q$, this cusp is equivalent to the cusp at infinity, while when $v = 1$, the cusp is equivalent to the cusp at zero. Note that if $q = 1$, so that $\chi$ is the trivial character, there is merely a single equivalence class of cusps, namely the cusp at infinity.

The scaling matrix $\sigma_{\aa} \in \SL_2(\R)$ for a singular cusp $\aa \sim 1/v$, $v \neq q$, can be chosen to be
\[\sigma_{\aa} \defeq \begin{pmatrix} \sqrt{w} & 0 \\ v \sqrt{w} & \dfrac{1}{\sqrt{w}}\end{pmatrix},\]
while for the cusp at infinity, we simply take $\sigma_{\infty}$ to be the identity. 

The Bruhat decomposition for $\sigma_{\aa}^{-1} \Gamma_0(q) \sigma_{\bb}$ \cite[Theorem 2.7]{Iwaniec} states that
\[\sigma_{\aa}^{-1} \Gamma_0(q) \sigma_{\bb} = \delta_{\aa\bb} \Omega_{\infty} \sqcup \bigsqcup_{c > 0} \bigsqcup_{d \hspace{-.2cm} \pmod{c}} \Omega_{d/c},\]
where $\delta_{\aa\bb} = 1$ if $\aa \sim \bb$ and $0$ otherwise, and
\begin{gather*}
\Omega_{\infty} \defeq \Gamma_{\infty} \omega_{\infty}, \qquad \omega_{\infty} = \begin{pmatrix} 1 & * \\ 0 & 1 \end{pmatrix} \in \sigma_{\aa}^{-1} \Gamma_0(q) \sigma_{\bb},	\\
\Omega_{d/c} \defeq \Gamma_{\infty} \omega_{d/c} \Gamma_{\infty}, \qquad \omega_{d/c} = \begin{pmatrix} * & * \\ c & d \end{pmatrix} \in \sigma_{\aa}^{-1} \Gamma_0(q) \sigma_{\bb} \quad \text{with $c > 0$,}
\end{gather*}
and $c,d$ runs over all real numbers such that $\sigma_{\aa}^{-1} \Gamma_0(q) \sigma_{\bb}$ contains $\left(\begin{smallmatrix} * & * \\ c & d \end{smallmatrix}\right)$. In particular, for the cusp at infinity we have the Bruhat decomposition
\[\sigma_{\infty}^{-1} \Gamma_0(q) \sigma_{\infty} = \Gamma_{\infty} \sqcup \bigsqcup_{\substack{c = 1 \\ c \equiv 0 \hspace{-.2cm} \pmod{q}}}^{\infty} \bigsqcup_{\substack{d \hspace{-.2cm} \pmod{c} \\ (c,d) = 1}} \Gamma_{\infty} \begin{pmatrix} * & * \\ c & d \end{pmatrix} \Gamma_{\infty}.\]
For $\aa \sim \infty$ and $\bb \sim 1/v$ a nonequivalent singular cusp with $1 \leq v < q$, $v$ dividing $q$, $vw = q$, and $(v,w) = 1$, and for any $\gamma = \left(\begin{smallmatrix} a & b \\ c & d \end{smallmatrix}\right) \in \Gamma_0(q)$, we have that
\[\sigma_{\infty}^{-1} \gamma \sigma_{\bb} = \begin{pmatrix} (a + bv)\sqrt{w} & \dfrac{b}{\sqrt{w}} \\ (c + dv)\sqrt{w} & \dfrac{d}{\sqrt{w}}\end{pmatrix},\]
and so
\begin{multline}\label{Bruhatinftyb}
\sigma_{\infty}^{-1} \Gamma_0(q) \sigma_{\bb} = \left\{\begin{pmatrix} a \sqrt{w} & \dfrac{b}{\sqrt{w}} \\ c \sqrt{w} & \dfrac{d}{\sqrt{w}} \end{pmatrix} \in \SL_2(\R) \colon \begin{pmatrix} a & b \\ c & d \end{pmatrix} \in \SL_2(\Z),\right.	\\
\left. \vphantom{\begin{pmatrix} \dfrac{b}{\sqrt{w}} \\ \dfrac{d}{\sqrt{w}} \end{pmatrix}} c \equiv 0 \hspace{-.2cm} \pmod{v}, \ d \equiv \frac{c}{v} \hspace{-.2cm} \pmod{w}, \ (c,d) = 1, \ (c,w) = 1\right\}.
\end{multline}
So the Bruhat decomposition in this case can be explicitly written in the form
\begin{equation}\label{Bruhatinftybexplicit}
\sigma_{\infty}^{-1} \Gamma_0(q) \sigma_{\bb} = \bigsqcup_{\substack{c = 1 \\ (c,w) = 1 \\ c \equiv 0 \hspace{-.2cm} \pmod{v}}}^{\infty} \bigsqcup_{\substack{d \hspace{-.2cm} \pmod{cw} \\ (cw,d) = 1 \\ d \equiv \frac{c}{v} \hspace{-.2cm} \pmod{w}}} \Gamma_{\infty} \begin{pmatrix} * & * \\ c\sqrt{w} & \dfrac{d}{\sqrt{w}} \end{pmatrix} \Gamma_{\infty}.
\end{equation}

\subsection{Eisenstein Series}

Given a primitive Dirichlet character $\chi$ modulo $q$ and a singular cusp $\aa$ of $\Gamma_0(q) \backslash \Hb$, we define the Eisenstein series $E_{\aa}\left(z, s, \chi\right)$ for $z \in \Hb$ and $\Re(s) > 1$ by
\[E_{\aa}\left(z, s, \chi\right) \defeq \sum_{\gamma \in \Gamma_{\aa} \backslash \Gamma_0(q)} \overline{\chi}(\gamma) j_{\sigma_{\aa}^{-1} \gamma}(z)^{-\kappa} \Im\left(\sigma_{\aa}^{-1} \gamma z\right)^s,\]
where $\kappa \in \{0,1\}$ is such that $\chi(-1) = (-1)^{\kappa}$, and for $\gamma = \left(\begin{smallmatrix} a & b \\ c & d \end{smallmatrix}\right) \in \SL_2(\R)$,
\[j_{\gamma}(z) \defeq \frac{cz + d}{|cz + d|} = e^{i \arg(cz + d)}.\]
The Eisenstein series associated to a singular cusp $\aa$ is independent of the choice of representative of $\aa$ and of the scaling matrix $\sigma_{\aa}$. For fixed $z \in \Hb$, the Eisenstein series $E_{\aa}\left(z, s, \chi\right)$ converges absolutely for $\Re(s) > 1$ and extends meromorphically to the entire complex plane with no poles on the closed right half-plane $\Re(s) \geq 1/2$ except at $s = 1$ when $q = 1$, so that $\chi$ is the trivial character.

For any $z \in \Hb$ and $\gamma_1, \gamma_2 \in \SL_2(\R)$, the $j$-factor satisfies the cocycle relation
\begin{equation}\label{cocycle}
j_{\gamma_1 \gamma_2}(z) = j_{\gamma_2}(z) j_{\gamma_1}(\gamma_2 z),
\end{equation}
while the Eisenstein series satisfies the automorphy condition
\begin{equation}\label{automorphy}
E_{\aa}\left(\gamma z, s, \chi\right) = \chi(\gamma) j_{\gamma}(z)^{\kappa} E_{\aa}\left(z, s, \chi\right)
\end{equation}
for any $\gamma \in \Gamma_0(q)$.

For any singular cusps $\aa,\bb$ of $\Gamma_0(q)$, one can show using the Bruhat decomposition that there exists a function $\varphi_{\aa\bb}(s,\chi)$ such that the constant term in the Fourier expansion for the function $j_{\sigma_{\bb}}(z)^{-\kappa} E_{\aa}\left(\sigma_{\bb} z, s, \chi\right)$ is
\[c_{\aa \bb}(z, s, \chi) \defeq \int_{0}^{1} j_{\sigma_{\bb}}(z)^{-\kappa} E_{\aa}\left(\sigma_{\bb} z, s, \chi\right) \, dx = \delta_{\aa \bb} y^s + \varphi_{\aa\bb}(s,\chi) y^{1 - s}.\]
The functions $\varphi_{\aa\bb}(s,\chi)$ are the entries of the scattering matrix associated to $\chi$. We will calculate $\varphi_{\aa\bb}(s,\chi)$ when $\aa \sim \infty$ for each nonsingular cusp $\bb$ of $\Gamma_0(q)$ with respect to $\chi$, and also find the rest of the Fourier coefficients of $E_{\infty}(z,s,\chi)$.

\subsection{Fourier Expansion of \texorpdfstring{$E_{\infty}(z,s,\chi)$}{E(z,s,\83\307)}}

\begin{lemma}\label{modtodivlemma}
Let $\chi$ be a primitive character modulo $q$. For $m \neq 0$ and $c \equiv 0 \pmod{q}$,
\[\sum_{\substack{d \hspace{-.2cm} \pmod{c} \\ (c,d) = 1}} \chi(d) e\left(\frac{md}{c}\right) = \chi(\sgn(m)) \tau(\chi) \sum_{d \mid \left(|m|,\frac{c}{q}\right)} d \overline{\chi}\left(\frac{|m|}{d}\right) \chi\left(\frac{c}{dq}\right) \mu\left(\frac{c}{dq}\right).\]
\end{lemma}

Here, as usual, we define $e(x) \defeq e^{2\pi i x}$ for $x \in \R$.

\begin{proof}
For $m$ positive, this is \cite[Lemma 3.1.3]{Miyake}. The result for $m$ negative follows by replacing $m$ with $|m|$ and $\chi$ with $\overline{\chi}$, then taking complex conjugates of both sides and using the fact that $\overline{\tau(\overline{\chi})} = \chi(-1) \tau(\chi)$.
\end{proof}

\begin{proposition}[{cf.\ \cite[Theorem 3.4]{Iwaniec}}]\label{Fouriercoeffsprop}
The Eisenstein series $E_{\infty}(z,s,\chi)$ has the Fourier expansion
\[E_{\infty}(z,s,\chi) = y^s + \varphi_{\infty \infty}(s,\chi) y^{1 - s} + \sum_{\substack{m = -\infty \\ m \neq 0}}^{\infty} \rho_{\infty}(m,s,\chi) W_{\sgn(m) \frac{\kappa}{2}, s - \frac{1}{2}}\left(4\pi |m| y\right) e(mx),\]
where $W_{\alpha,\nu}(y)$ is the Whittaker function,
\[\varphi_{\infty \infty}(s,\chi) = \begin{dcases*}
\sqrt{\pi} \frac{\Gamma\left(s - \frac{1}{2}\right)}{\Gamma(s)} \frac{\zeta(2s - 1)}{\zeta(2s)} & if $q = 1$,	\\
0 & if $q \geq 2$,
\end{dcases*}\]
and for $m \neq 0$,
\[\rho_{\infty}(m,s,\chi) = \frac{\chi(\sgn(m)) i^{-\kappa} \tau(\overline{\chi}) \pi^s |m|^{s - 1}}{q^{2s} \Gamma\left(s + \sgn(m) \frac{\kappa}{2}\right) L(2s, \overline{\chi})} \sigma_{1 - 2s}(|m|,\chi),\]
where $\tau(\chi)$ is the Gauss sum of $\chi$ and
\[\sigma_s(m,\chi) \defeq \sum_{d \mid m} d^s \chi\left(\frac{m}{d}\right).\]
\end{proposition}

Note in particular that if $\kappa = 0$, so that $\chi$ is even, the Whittaker function is simply
\[W_{0, s - \frac{1}{2}}\left(4\pi |m| y\right) = \sqrt{4|m|y} K_{s - \frac{1}{2}}(2\pi |m| y),\]
where $K_{\nu}(y)$ is the $K$-Bessel function. On the other hand, if $\kappa = 1$, so that $\chi$ is odd, and we set $s = 1/2$, then
\[W_{\sgn(m) \frac{\kappa}{2}, 0}\left(4\pi |m| y\right) = \begin{dcases*}
\sqrt{4\pi |m| y} e^{-2\pi|m| y} & if $m > 0$,	\\
\sqrt{4\pi |m| y} e^{2\pi|m| y} \int_{4\pi|m|y}^{\infty} \frac{e^{-u}}{u} \, du & if $m < 0$.
\end{dcases*}\]

\begin{proof}
Via the Bruhat decomposition \eqref{Bruhatinftybexplicit}, $E_{\infty}(z,s,\chi)$ is equal to
\[y^s + \sum_{\substack{c = 1 \\ c \equiv 0 \hspace{-.2cm} \pmod{q}}}^{\infty} \sum_{\substack{d \hspace{-.2cm} \pmod{c} \\ (c,d) = 1}} \overline{\chi}(d) \sum_{n = -\infty}^{\infty} \left(\frac{c(z + n) + d}{|c(z + n) + d|}\right)^{-\kappa} \frac{y^s}{|c(z + n) + d|^{2s}}.\]
So if $m = 0$, the zeroeth Fourier coefficient of $E_{\infty}(z,s,\chi)$ is
\begin{multline*}
y^s + \sum_{\substack{c = 1 \\ c \equiv 0 \hspace{-.2cm} \pmod{q}}}^{\infty} \sum_{\substack{d \hspace{-.2cm} \pmod{c} \\ (c,d) = 1}} \overline{\chi}(d) \int_{-\infty}^{\infty} \left(\frac{cz + d}{|cz + d|}\right)^{-\kappa} \frac{y^s}{|cz + d|^{2s}} \, dx	\\
= y^s + y^{1 - s} \int_{-\infty}^{\infty} \left(\frac{t + i}{|t + i|}\right)^{-\kappa} \frac{1}{|t + i|^{2s}} \, dt \sum_{\substack{c = 1 \\ c \equiv 0 \hspace{-.2cm} \pmod{q}}}^{\infty} \frac{1}{c^{2s}} \sum_{\substack{d \hspace{-.2cm} \pmod{c} \\ (c,d) = 1}} \overline{\chi}(d)
\end{multline*}
by the change of variables $x \mapsto yt - \frac{d}{c}$. From \cite[(8.381.1)]{Gradshteyn}, we have that
\[\int_{-\infty}^{\infty} \left(\frac{t + i}{|t + i|}\right)^{-\kappa} \frac{1}{|t + i|^{2s}} \, dt = i^{-\kappa} \sqrt{\pi} \frac{\Gamma\left(\frac{1}{2}(2s - 1 + \kappa)\right)}{\Gamma\left(\frac{1}{2}(2s + \kappa)\right)},\]
while for $c \equiv 0 \pmod{q}$, the fact that $\chi$ is primitive implies that 
\[\sum_{\substack{c = 1 \\ c \equiv 0 \hspace{-.2cm} \pmod{q}}}^{\infty} \frac{1}{c^{2s}} \sum_{\substack{d \hspace{-.2cm} \pmod{c} \\ (c,d) = 1}} \overline{\chi}(d) = \begin{dcases*}
\sum_{c = 1}^{\infty} \frac{\varphi(c)}{c^{2s}} = \frac{\zeta(2s - 1)}{\zeta(2s)} & if $q = 1$,	\\
0 & if $q \geq 2$.
\end{dcases*}\]

If $m \neq 0$, on the other hand, then the $m$-th Fourier coefficient is
\begin{multline*}
\sum_{\substack{c = 1 \\ c \equiv 0 \hspace{-.2cm} \pmod{q}}}^{\infty} \sum_{\substack{d \hspace{-.2cm} \pmod{c} \\ (c,d) = 1}} \overline{\chi}(d) \int_{-\infty}^{\infty} \left(\frac{cz + d}{|cz + d|}\right)^{-\kappa} \frac{y^s}{|cz + d|^{2s}} e(-mx) \, dx	\\
= y^{1 - s} \int_{-\infty}^{\infty} \left(\frac{t + i}{|t + i|}\right)^{-\kappa} \frac{e(-myt)}{|t + i|^{2s}} \, dt \sum_{\substack{c = 1 \\ c \equiv 0 \hspace{-.2cm} \pmod{q}}}^{\infty} \frac{1}{c^{2s}} \sum_{\substack{d \hspace{-.2cm} \pmod{c} \\ (c,d) = 1}} \overline{\chi}(d)  e\left(\frac{md}{c}\right)
\end{multline*}
again by the change of variables $x \mapsto yt - \frac{d}{c}$. Moreover, \cite[(3.384.9)]{Gradshteyn} implies that
\[\int_{-\infty}^{\infty} \left(\frac{t + i}{|t + i|}\right)^{-\kappa} \frac{e(-myt)}{|t + i|^{2s}} \, dt = \frac{i^{-\kappa} \pi^s |m|^{s - 1} y^{s - 1}}{\Gamma\left(s + \sgn(m) \frac{\kappa}{2}\right)} W_{\sgn(m) \frac{\kappa}{2}, s - \frac{1}{2}}\left(4\pi |m| y\right),\]
and via \hyperref[modtodivlemma]{Lemma \ref*{modtodivlemma}},
\begin{align*}
\hspace{2cm} & \hspace{-2cm} \sum_{\substack{c = 1 \\ c \equiv 0 \hspace{-.2cm} \pmod{q}}}^{\infty} \frac{1}{c^{2s}} \sum_{\substack{d \hspace{-.2cm} \pmod{c} \\ (c,d) = 1}} \overline{\chi}(d) e\left(\frac{md}{c}\right)	\\
& = \chi(\sgn(m)) \tau(\overline{\chi}) \sum_{d \mid |m|} d \chi\left(\frac{|m|}{d}\right) \sum_{\substack{c = 1 \\ c \equiv 0 \hspace{-.2cm} \pmod{dq}}}^{\infty} \frac{\overline{\chi}\left(\frac{c}{dq}\right) \mu\left(\frac{c}{dq}\right)}{c^{2s}}	\\
& = \chi(\sgn(m)) \frac{\tau(\overline{\chi})}{q^{2s}} \sum_{d \mid |m|} d^{1 - 2s} \chi\left(\frac{|m|}{d}\right) \sum_{n = 1}^{\infty} \frac{\overline{\chi}(n) \mu(n)}{n^{2s}}	\\
& = \chi(\sgn(m)) \frac{\tau(\overline{\chi})}{q^{2s} L(2s, \overline{\chi})} \sigma_{1 - 2s}(|m|,\chi)
\end{align*}
where we have let $c = dqn$. We thereby obtain the desired identity.
\end{proof}

\begin{proposition}
Suppose that $q \geq 2$. Then $\varphi_{\infty \bb}(s,\chi)$ vanishes unless $\bb \sim 1$, in which case
\begin{equation}\label{scatteringLambda}
\varphi_{\infty 1}(s,\chi) = \frac{\overline{\tau(\chi)}}{q^s} \frac{\Lambda(2 - 2s,\chi)}{\Lambda(2s,\overline{\chi})},
\end{equation}
where
\begin{equation}\label{Lambdadef}
\Lambda(s,\chi) \defeq \left(\frac{\pi}{q}\right)^{-\frac{s + \kappa}{2}} \Gamma\left(\frac{s + \kappa}{2}\right) L(s,\chi),
\end{equation}
is the completed Dirichlet $L$-function. In particular,
\begin{equation}\label{scatteringunitary}
\left|\varphi_{\infty 1}\left(\frac{1}{2} + it,\chi\right)\right| = 1.
\end{equation}
\end{proposition}

\begin{proof}
The fact that $\varphi_{\infty \bb}(s,\chi) = 0$ when $\bb$ is the cusp at infinity follows from \hyperref[Fouriercoeffsprop]{Proposition \ref*{Fouriercoeffsprop}}. For the entries of the scattering matrix at other cusps, we use \eqref{cocycle} to write
\[E_{\aa}\left(\sigma_{\bb} z, s, \chi\right) = j_{\sigma_{\bb}}(z)^{\kappa} \sum_{\gamma \in \Gamma_{\infty} \backslash \sigma_{\aa}^{-1} \Gamma_0(q) \sigma_{\bb}} \overline{\chi}(\sigma_{\aa} \gamma \sigma_{\bb}^{-1}) j_{\gamma}(z)^{-\kappa} \Im(\gamma z)^s.\]
The singular cusp $\bb$ is equivalent to $1/v$ for some divisor $v$ of $q$ with $v < q$, $vw = q$, and $(v,w) = 1$. Given a matrix
\[\gamma = \begin{pmatrix} a \sqrt{w} & \dfrac{b}{\sqrt{w}} \\ c\sqrt{w} & \dfrac{d}{\sqrt{w}} \end{pmatrix}\]
in $\sigma_{\infty}^{-1} \Gamma_0(q) \sigma_{\bb}$ as in \eqref{Bruhatinftyb}, we have that
\[\sigma_{\infty} \gamma \sigma_{\bb}^{-1} 
= \begin{pmatrix} a - bv & b	\\ c - dv & d \end{pmatrix},\]
and so as $d \equiv \frac{c}{v} \pmod{w}$,
\[\overline{\chi}\left(\sigma_{\infty} \gamma \sigma_{\bb}^{-1}\right) = \overline{\chi_v}(d) \overline{\chi_w}\left(\frac{c}{v}\right),\]
where we have decomposed the primitive character $\chi$ modulo $q$ into the product of primitive characters $\chi_v$ modulo $v$ and $\chi_w$ modulo $w$. From this and \eqref{Bruhatinftybexplicit}, we see that $j_{\sigma_{\bb}}(z)^{-\kappa} E_{\infty}\left(\sigma_{\bb} z,s,\chi\right)$ is equal to
\begin{multline*}
\sum_{\substack{c = 1 \\ (c,w) = 1 \\ c \equiv 0 \hspace{-.2cm} \pmod{v}}}^{\infty} \overline{\chi_w}\left(\frac{c}{v}\right) \sum_{\substack{d \hspace{-.2cm} \pmod{cw} \\ (cw,d) = 1 \\ d \equiv \frac{c}{v} \hspace{-.2cm} \pmod{w}}} \overline{\chi_v}(d)	\\
\times \sum_{n = -\infty}^{\infty} \left(\frac{c(z + n)\sqrt{w} + \dfrac{d}{\sqrt{w}}}{\left|c(z + n)\sqrt{w} + \dfrac{d}{\sqrt{w}}\right|}\right)^{-\kappa} \frac{y^s}{\left|c(z + n)\sqrt{w} + \dfrac{d}{\sqrt{w}}\right|^{2s}},
\end{multline*}
and so integrating from $0$ to $1$ with respect to $x$, making the change of variables $x \mapsto yt - \frac{d}{cw}$, and dividing by $y^{1 - s}$ yields
\[\varphi_{\infty \bb}(s,\chi) = \frac{1}{w^s} \int_{-\infty}^{\infty} \left(\frac{t + i}{|t + i|}\right)^{-\kappa} \frac{1}{|t + i|^{2s}} \, dt \sum_{\substack{c = 1 \\ (c,w) = 1 \\ c \equiv 0 \hspace{-.2cm} \pmod{v}}}^{\infty} \frac{\overline{\chi_w}\left(\frac{c}{v}\right)}{c^{2s}} \sum_{\substack{d \hspace{-.2cm} \pmod{cw} \\ (cw,d) = 1 \\ d \equiv \frac{c}{v} \hspace{-.2cm} \pmod{w}}} \overline{\chi_v}(d).\]
From \cite[(8.381.1)]{Gradshteyn}, the integral is equal to
\[i^{-\kappa} \sqrt{\pi} \frac{\Gamma\left(\frac{1}{2}(2s - 1 + \kappa)\right)}{\Gamma\left(\frac{1}{2}(2s + \kappa)\right)}.\]
To evaluate the sum over $d$, we write $d = \overline{v} c + w d'$, where $\overline{v} v \equiv 1 \pmod{w}$ and $(d',c) = 1$. This allows us to replace the sum over $d$ with a sum over $d'$ modulo $c$ with $(c,d') = 1$, so that
\[\sum_{\substack{d \hspace{-.2cm} \pmod{cw} \\ (cw,d) = 1 \\ d \equiv \frac{c}{v} \hspace{-.2cm} \pmod{w}}} \overline{\chi_v}(d) = \overline{\chi_v}(w) \sum_{\substack{d' \hspace{-.2cm} \pmod{c} \\ (c,d') = 1}} \overline{\chi_v}(d')\]
by the fact that $c \equiv 0 \pmod{v}$.

If $\overline{\chi_v}$ is nonprincipal, this sum vanishes, and as $\chi$ is a primitive character, $\overline{\chi_v}$ can only be the principal character if $v = 1$; consequently, $\varphi_{\infty \bb}(s,\chi)$ vanishes if $\bb$ is inequivalent to the cusp at $1$.

If $\bb \sim 1$, so that $v = 1$ and $w = q$, then this sum over $d'$ is merely $\varphi(c)$, and so
\[\sum_{\substack{c = 1 \\ (c,w) = 1 \\ c \equiv 0 \hspace{-.2cm} \pmod{v}}}^{\infty} \frac{\overline{\chi_w}\left(\frac{c}{v}\right)}{c^{2s}} \sum_{\substack{d \hspace{-.2cm} \pmod{cw} \\ (cw,d) = 1 \\ d \equiv \frac{c}{v} \hspace{-.2cm} \pmod{w}}} \overline{\chi_v}(d) = \sum_{c = 1}^{\infty} \frac{\varphi(c) \overline{\chi}(c)}{c^{2s}} = \frac{L(2s - 1, \overline{\chi})}{L(2s,\overline{\chi})}.\]
Using the definition of the completed Dirichlet $L$-function together with the fact that it satisfies the functional equation
\[\Lambda(s,\chi) = \frac{\tau(\chi)}{i^{\kappa} \sqrt{q}} \Lambda(1 - s,\overline{\chi}),\]
we see that we may write
\[\varphi_{\infty 1}(s,\chi) = \frac{i^{-\kappa}}{q^{s - \frac{1}{2}}} \frac{\Lambda(2s - 1,\overline{\chi})}{\Lambda(2s,\overline{\chi})} = \frac{\overline{\tau(\chi)}}{q^s} \frac{\Lambda(2 - 2s,\chi)}{\Lambda(2s,\overline{\chi})}.\]
As $\overline{\Lambda(s,\chi)} = \Lambda(\overline{s},\overline{\chi})$ and $|\tau(\chi)| = \sqrt{q}$, the result follows.
\end{proof}

\section{\texorpdfstring{Maa\ss{}--Selberg Relation}{Maa\80\337{}\9040\023{}Selberg Relation}}\label{MSrelsect}

For $z \in \Hb$ and $T \geq 1$, we define the truncated Eisenstein series
\begin{equation}\label{truncatedEisen}
\Lambda^T E_{\aa}(z,s,\chi) \defeq E_{\aa}(z,s,\chi) - \sum_{\cc} \sum_{\substack{\gamma \in \Gamma_{\cc} \backslash \Gamma_0(q) \\ \Im(\sigma_{\cc}^{-1} \gamma z) > T}} \overline{\chi}(\gamma) j_{\sigma_{\cc}^{-1} \gamma}(z)^{-\kappa} c_{\aa\cc}(\sigma_{\cc}^{-1} \gamma z,s,\chi),
\end{equation}
where the summation over $\cc$ is over all singular cusps of $\Gamma_0(q) \backslash \Hb$. It is not difficult to see that $\Lambda^T E_{\aa}(z,s,\chi)$ satisfies the automorphy condition
\begin{equation}\label{truncatedautomorphy}
\Lambda^T E_{\aa}\left(\gamma z, s, \chi\right) = \chi(\gamma) j_{\gamma}(z)^{\kappa} \Lambda^T E_{\aa}\left(z, s, \chi\right)
\end{equation}
for any $\gamma \in \Gamma_0(q)$. We will show that, unlike $E_{\aa}(z,s,\chi)$, the function $\Lambda^T E_{\aa}(z,s,\chi)$ is square-integrable on $\Gamma_0(q) \backslash \Hb$, and give an explicit expression for the resulting integral.

\begin{lemma}\label{ImzImgammaz}
Let $\bb$ and $\cc$ be singular cusps of $\Gamma_0(q) \backslash \Hb$, and let $\gamma \in \sigma_{\cc}^{-1} \Gamma_0(q) \sigma_{\bb}$. Then for any $z = x + iy \in \Hb$, we have that $\Im(z) \Im(\gamma z) \leq 1$ if $\bb$ and $\cc$ are inequivalent or if $\bb$ and $\cc$ are equivalent but $\gamma \notin \Gamma_{\infty} \omega_{\infty}$. If $\bb$ and $\cc$ are equivalent and $\gamma \in \Gamma_{\infty} \omega_{\infty}$, then $\Im(\gamma z) = \Im(z)$.
\end{lemma}

\begin{proof}
We deal with the cases where neither $\bb$ nor $\cc$ are equivalent to the cusp at infinity; when $\bb \sim \infty$ or $\cc \sim \infty$, the proof is similar but simpler. Let $\bb \sim 1/v$ and $\cc \sim 1/v'$, $1 \leq v,v' < q$, with $w,w'$ such that $vw = v'w' = q$. For $\left(\begin{smallmatrix} a & b \\ c & d \end{smallmatrix}\right) \in \Gamma_0(q)$, we have that
\[\sigma_{\cc}^{-1} \begin{pmatrix} a & b \\ c & d \end{pmatrix} \sigma_{\bb} 
= \begin{pmatrix} (a + bv) \sqrt{\dfrac{w}{w'}} & \dfrac{b}{\sqrt{w'w}}	\\ (c - av' + dv - bv'v) \sqrt{w'w} & (d - bv') \sqrt{\dfrac{w'}{w}} \end{pmatrix}.\]
So for
\[\gamma = \begin{pmatrix} * & * \\ C \sqrt{w'w} & D \sqrt{\dfrac{w'}{w}} \end{pmatrix} \in \sigma_{\cc}^{-1} \Gamma_0(q) \sigma_{\bb},\]
where $C = c - av' + dv - bv'v$ and $D = d - bv'$ are integers, we have that
\[\Im(\gamma z) = \frac{1}{w'w} \frac{y}{(C x + Dw^{-1})^2 + C^2 y^2}.\]
By the Bruhat decomposition, if $\bb$ and $\cc$ are inequivalent, then $C \sqrt{w'w}$ must be nonzero, and so $C^2 \geq 1$. In particular, if $\bb$ and $\cc$ are inequivalent, then
\[\Im(z) \Im(\gamma z) \leq \frac{1}{w'w} \leq 1.\]
If $\bb$ and $\cc$ are equivalent and $\gamma \notin \Gamma_{\infty} \omega_{\infty}$, then again $C \sqrt{w'w} \neq 0$, and the same result holds. Finally, if $\bb$ and $\cc$ are equivalent and $\gamma \in \Gamma_{\infty} \omega_{\infty}$, then it is clear that $\Im(\gamma z) = \Im(z)$.
\end{proof}

\begin{corollary}\label{LambdaTcuspzone}
If $\Im(z) > T \geq 1$, then for any singular cusp $\bb$, we have that
\[\Lambda^T E_{\aa}(\sigma_{\bb} z,s,\chi) = E_{\aa}(\sigma_{\bb} z,s,\chi) - j_{\sigma_{\bb}}(z)^{\kappa} c_{\aa\bb}(z,s,\chi).\]
\end{corollary}

\begin{proof}
From the definition of $\Lambda^T E_{\aa}(z,s,\chi)$ and \eqref{cocycle}, we must show that for any singular cusp $\cc$ and $\gamma \in \Gamma_{\cc} \backslash \Gamma_0(q)$ that the inequalities $\Im(z) > T$ and $\Im(\sigma_{\cc}^{-1} \gamma \sigma_{\bb} z) > T$ are simultaneously satisfied only when $\cc \sim \bb$ and $\gamma = \omega_{\infty}$. This is equivalent to showing that if $\gamma \in \Gamma_{\infty} \backslash \sigma_{\cc}^{-1} \Gamma_0(q) \sigma_{\bb}$ is such that $\Im(z) > T$ and $\Im(\gamma z) > T$, then $\cc \sim \bb$ and $\gamma = \omega_{\infty}$, which follows immediately from \hyperref[ImzImgammaz]{Lemma \ref*{ImzImgammaz}}.
\end{proof}

With these results in hand, we can prove the following Maa\ss{}--Selberg relation.

\begin{proposition}
For any two singular cusps $\aa,\bb$, $T \geq 1$, and $s \neq \overline{r}$, $s + \overline{r} \neq 1$,
\begin{multline*}
\int_{\Gamma_0(q) \backslash \Hb} \Lambda^T E_{\aa}(z,s,\chi) \overline{\Lambda^T E_{\bb}(z,r,\chi)} \, d\mu(z)	\\
= \overline{\varphi_{\bb\aa}(r,\chi)} \frac{T^{s - \overline{r}}}{s - \overline{r}} + \varphi_{\aa\bb}(s,\chi) \frac{T^{\overline{r} - s}}{\overline{r} - s} + \delta_{\aa \bb} \frac{T^{s + \overline{r} - 1}}{s + \overline{r} - 1}	\\
+ \sum_{\cc} \varphi_{\aa\cc}(s,\chi) \overline{\varphi_{\bb\cc}(r,\chi)} \frac{T^{1 - s - \overline{r}}}{1 - s - \overline{r}},
\end{multline*}
where the sum is over singular cusps $\cc$. Here $d\mu(z) = \dfrac{dx \, dy}{y^2}$ is the $\SL_2(\R)$-invariant measure on $\Hb$.
\end{proposition}

\begin{proof}
We initially assume that $\Re(s), \Re(r) > 1$ with $\Re(s) - \Re(r) > 1$; the identity then extends to all $s,r \in \C$ with $s \neq \overline{r}$ and $s + \overline{r} \neq 1$ by analytic continuation.

We first show that
\[\int_{\Gamma_0(q) \backslash \Hb} \Lambda^T E_{\aa}(z,s,\chi) \left(\overline{\Lambda^T E_{\bb}(z,r,\chi)} - \overline{E_{\bb}(z,r,\chi)}\right) \, d\mu(z) = 0.\]
Indeed, the left-hand side is equal to
\[\sum_{\cc} \int_{\Gamma_0(q) \backslash \Hb} \Lambda^T E_{\aa}(z,s,\chi) \sum_{\substack{\gamma \in \Gamma_{\cc} \backslash \Gamma_0(q) \\ \Im(\sigma_{\cc}^{-1} \gamma z) > T}} \chi(\gamma) \overline{j_{\sigma_{\cc}^{-1} \gamma}(z)^{-\kappa} c_{\bb\cc}(\sigma_{\cc}^{-1} \gamma z,r,\chi)} \, d\mu(z),\]
which, by \eqref{cocycle} and \eqref{truncatedautomorphy}, is equal to
\[- \sum_{\cc} \int_{\Gamma_0(q) \backslash \Hb} \sum_{\substack{\gamma \in \Gamma_{\cc} \backslash \Gamma_0(q) \\ \Im(\sigma_{\cc}^{-1} \gamma z) > T}} \overline{c_{\bb\cc}(\sigma_{\cc}^{-1} \gamma z,r,\chi)} j_{\sigma_{\cc}}(\sigma_{\cc}^{-1} \gamma z)^{-\kappa} \Lambda^T E_{\aa}(\gamma z,s,\chi) \, d\mu(z),\]
and this integral can be unfolded to yield
\[- \sum_{\cc} \int_{T}^{\infty} \int_{0}^{1} \overline{c_{\bb\cc}(z,r,\chi)} j_{\sigma_{\cc}}(z)^{-\kappa} \Lambda^T E_{\aa}(\sigma_{\cc} z,s,\chi) \, \frac{dx \, dy}{y^2}.\]
But $\overline{c_{\bb\cc}(z,r,\chi)}$ is independent of $x$, while for $\Im(z) > T \geq 1$, the zeroeth Fourier coefficient of the function $j_{\sigma_{\cc}}(z)^{-\kappa} \Lambda^T E_{\aa}(\sigma_{\cc} z,s,\chi)$ vanishes via \hyperref[LambdaTcuspzone]{Corollary \ref*{LambdaTcuspzone}}, and so this vanishes. Consequently,
\[\int_{\Gamma_0(q) \backslash \Hb} \Lambda^T E_{\aa}(z,s,\chi) \overline{\Lambda^T E_{\bb}(z,r,\chi)} \, d\mu(z) = \int_{\Gamma_0(q) \backslash \Hb} \Lambda^T E_{\aa}(z,s,\chi) \overline{E_{\bb}(z,r,\chi)} \, d\mu(z).\]
The right-hand side can be written as
\begin{align*}
& \int_{\Gamma_0(q) \backslash \Hb} \left(\sum_{\gamma \in \Gamma_{\aa} \backslash \Gamma_0(q)} \overline{\chi}(\gamma) j_{\sigma_{\aa}^{-1} \gamma}(z)^{-\kappa} \Im(\sigma_{\aa}^{-1} \gamma z)^s \overline{E_{\bb}(z,r,\chi)}\right.	\\
& \hspace{2cm} \left. - \sum_{\cc} \sum_{\substack{\gamma \in \Gamma_{\cc} \backslash \Gamma_0(q) \\ \Im(\sigma_{\cc}^{-1} \gamma z) > T}} \overline{\chi}(\gamma) j_{\sigma_{\cc}^{-1} \gamma}(z)^{-\kappa} c_{\aa\cc}(\sigma_{\cc}^{-1} \gamma z,s,\chi) \overline{E_{\bb}(z,r,\chi)}\right) \, d\mu(z)	\\
& = \int_{\Gamma_0(q) \backslash \Hb} \sum_{\substack{\gamma \in \Gamma_{\aa} \backslash \Gamma_0(q) \\ \Im(\sigma_{\aa}^{-1} \gamma z) \leq T}} \overline{\chi}(\gamma) j_{\sigma_{\aa}^{-1} \gamma}(z)^{-\kappa} \Im(\sigma_{\aa}^{-1} \gamma z)^s \overline{E_{\bb}(z,r,\chi)} \, d\mu(z)	\\
& \qquad + \int_{\Gamma_0(q) \backslash \Hb} \sum_{\substack{\gamma \in \Gamma_{\aa} \backslash \Gamma_0(q) \\ \Im(\sigma_{\aa}^{-1} \gamma z) > T}} \overline{\chi}(\gamma) j_{\sigma_{\aa}^{-1} \gamma}(z)^{-\kappa} \varphi_{\aa\aa}(s,\chi) \Im(\sigma_{\aa}^{-1} \gamma z)^{1 - s} \overline{E_{\bb}(z,r,\chi)} \, d\mu(z)	\\
& \qquad - \sum_{\cc \neq \aa} \int_{\Gamma_0(q) \backslash \Hb} \sum_{\substack{\gamma \in \Gamma_{\cc} \backslash \Gamma_0(q) \\ \Im(\sigma_{\cc}^{-1} \gamma z) > T}} \overline{\chi}(\gamma) j_{\sigma_{\cc}^{-1} \gamma}(z)^{-\kappa} c_{\aa\cc}(\sigma_{\cc}^{-1} \gamma z,s,\chi) \overline{E_{\bb}(z,r,\chi)} \, d\mu(z).
\end{align*}
By \eqref{cocycle} and \eqref{automorphy}, the first term is
\[\int_{\Gamma_0(q) \backslash \Hb} \sum_{\substack{\gamma \in \Gamma_{\aa} \backslash \Gamma_0(q) \\ \Im(\sigma_{\aa}^{-1} \gamma z) \leq T}} \Im(\sigma_{\aa}^{-1} \gamma z)^s \overline{j_{\sigma_{\aa}}(\sigma_{\aa}^{-1} \gamma z)^{-\kappa}  E_{\bb}(\gamma z,r,\chi)} \, d\mu(z),\]
and upon unfolding the integral, this becomes
\begin{align*}
\int_{0}^{T} \int_{0}^{1} y^s \overline{j_{\sigma_{\aa}}(z)^{-\kappa} E_{\bb}(\sigma_{\aa} z,r,\chi)} \, \frac{dx \, dy}{y^2} & = \int_{0}^{T} y^s \overline{c_{\bb\aa}(z,r,\chi)} \, \frac{dy}{y^2}	\\
& = \delta_{\aa \bb} \frac{T^{s + \overline{r} - 1}}{s + \overline{r} - 1} + \overline{\varphi_{\bb\aa}(r,\chi)} \frac{T^{s - \overline{r}}}{s - \overline{r}}.
\end{align*}
Similarly, the second term is
\[\int_{T}^{\infty} \varphi_{\aa\aa}(s,\chi) y^{1 - s} \overline{c_{\bb\aa}(z,s,\chi)} \, \frac{dy}{y^2} = \delta_{\aa\bb} \varphi_{\aa\bb}(s,\chi) \frac{T^{\overline{r} - s}}{\overline{r} - s} + \varphi_{\aa\aa}(s,\chi) \overline{\varphi_{\bb\aa}(r,\chi)} \frac{T^{1 - s - \overline{r}}}{1 - s - \overline{r}},\]
and the third term is
\begin{multline*}
-\sum_{\cc \neq \aa} \int_{T}^{\infty}  c_{\aa\cc}(z,s,\chi) \overline{c_{\bb\cc}(z,r,\chi)} \, \frac{dy}{y^2}	\\
= (1 - \delta_{\aa\bb})\varphi_{\aa\bb}(s,\chi) \frac{T^{\overline{r} - s}}{\overline{r} - s} + \sum_{\cc \neq \aa} \varphi_{\aa\cc}(s,\chi) \overline{\varphi_{\bb\cc}(r,\chi)} \frac{T^{1 - s - \overline{r}}}{1 - s - \overline{r}}.
\end{multline*}
Combining these identities yields the result.
\end{proof}

\begin{corollary}
For $T \geq 1$ and $t \in \R$, we have that
\[\int_{\Gamma_0(q) \backslash \Hb} \left|\Lambda^T E_{\infty}\left(z,\frac{1}{2} + it,\chi\right)\right|^2 \, d\mu(z) = 2 \log T - \Re\left(\frac{\varphi_{\infty 1}'}{\varphi_{\infty 1}} \left(\frac{1}{2} + it,\chi\right)\right).\]
\end{corollary}

\begin{proof}
We take $\aa \sim \bb \sim \infty$ and $s = r = 1/2 + it + \e$ with $\e > 0$ in the Maa\ss{}--Selberg relation to obtain
\[\int_{\Gamma_0(q) \backslash \Hb} \left|\Lambda^T E_{\infty}\left(z,\frac{1}{2} + it + \e,\chi\right)\right|^2 \, d\mu(z) = \frac{T^{2\e}}{2\e} - \left|\varphi_{\infty 1}\left(\frac{1}{2} + it + \e,\chi\right)\right|^2 \frac{T^{-2\e}}{2\e}.\]
The result then follows by taking the limit as $\e$ tends to zero and using the Taylor expansions
\begin{align*}
T^{2\e} & = 1 + 2\e \log T + O\left(\e^2\right),	\\
\varphi_{\infty 1}\left(\frac{1}{2} + it + \e,\chi\right) & = \varphi_{\infty 1}\left(\frac{1}{2} + it,\chi\right) + \e \varphi_{\infty 1}'\left(\frac{1}{2} + it,\chi\right) + O\left(\e^2\right),
\end{align*}
together with \eqref{scatteringunitary}.
\end{proof}

\begin{remark}
This proof of the Maa\ss{}--Selberg relation is via unfolding as in \cite[Section 4]{Arthur}, and makes use of the Arthur truncation $\Lambda^T E_{\aa}(z,s,\chi)$ of the Eisenstein series $E_{\aa}(z,s,\chi)$ given by \eqref{truncatedEisen}; cf.\ \cite[Section 1]{Arthur}. One can instead prove the Maa\ss{}--Selberg relation without recourse to the automorphy of the truncated Eisenstein series by only defining $\Lambda^T E_{\aa}(z,s,\chi)$ within a fundamental domain of $\Gamma_0(q) \backslash \Hb$. Let
\[\FF \supset \left\{z \in \Hb \colon 0 < \Re(z) < 1, \ \Im(z) \geq 1\right\}\]
be the usual fundamental domain of $\Gamma_0(q) \backslash \Hb$, and for each singular cusp $\aa$, we define the cuspidal zone
\[\FF_{\aa}(T) \defeq \left\{z \in \FF \colon 0 < \Re\left(\sigma_{\aa}^{-1} z\right) < 1, \ \Im\left(\sigma_{\aa}^{-1} z\right) \geq T\right\}\]
for $T \geq 1$; note that any two cuspidal zones will be disjoint provided that $T$ is sufficiently large. Then from \hyperref[ImzImgammaz]{Lemma \ref*{ImzImgammaz}}, we have that for $T \geq 1$,
\[\Lambda^T E_{\aa}\left(z, s, \chi\right) = \begin{dcases*}
E_{\aa}\left(z, s, \chi\right) & \hspace{-3cm} if $z \in \FF \setminus \bigcup_{\cc} \FF_{\cc}(T)$,	\\
E_{\aa}\left(z, s, \chi\right) - \sum_{\cc \in A} \left(\delta_{\aa\cc} \Im\left(\sigma_{\cc}^{-1} z\right)^s + \varphi_{\aa\cc}(s,\chi) \Im\left(\sigma_{\cc}^{-1} z\right)^{1 - s}\right) &	\\
& \hspace{-3cm} if $z \in \bigcap_{\cc \in A} \FF_{\cc}(T)$,
\end{dcases*}\]
where $A$ is any subset of the set of singular cusps. The Maa\ss{}--Selberg relation may then be proved using Green's theorem along the same lines as the proof of \cite[Proposition 6.8]{Iwaniec}.
\end{remark}

\section{Upper Bounds and Lower Bounds for the Integral \texorpdfstring{$\II$}{I}}

For $\eta \leq 1$, we consider the integral
\[\II = \II(\chi,\eta,T) \defeq \int_{\eta}^{\infty} \int_{0}^{1} \left|\Lambda^T E_{\infty}\left(z,\frac{1}{2},\chi\right)\right|^2 \, \frac{dx \, dy}{y^2}.\]
Our goal is to find upper and lower bounds for this integral: upper bounds via the Maa\ss{}--Selberg relation and lower bounds via Parseval's identity and the Brun--Titchmarsh inequality. Combining these bounds will yield lower bounds for $L(1,\chi)$.

\subsection{Upper Bounds for \texorpdfstring{$\II$}{I}}

\begin{proposition}
For $\eta \ll 1/q$ and $T \geq 1$, we have that
\[\II \ll \frac{\log q \log qT}{q \eta |L(1,\chi)|}.\]
\end{proposition}

\begin{proof}
By folding the integral, one can write
\[\II = \int_{\Gamma_0(q) \backslash \Hb} N_q(z,\eta) \left|\Lambda^T E_{\infty}\left(z,\frac{1}{2},\chi\right)\right|^2 \, d\mu(z),\]
where for $\eta \leq 1$,
\[N_q(z,\eta) \defeq \# \left\{\gamma \in \Gamma_{\infty} \backslash \Gamma_0(q) \colon \Im(\gamma z) > \eta\right\}.\]
The Maa\ss{}--Selberg relation then implies the upper bound
\[\II \leq \sup_{z \in \Gamma_0(q) \backslash \Hb} N_q(z,\eta) \left(2\log T - \Re\left(\frac{\varphi_{\infty 1}'}{\varphi_{\infty 1}} \left(\frac{1}{2},\chi\right)\right)\right).\]
From \cite[Lemma 2.10]{Iwaniec}, we have the bound
\[N_q(z,\eta) < 1 +  \frac{10}{q \eta}.\]
By taking logarithmic derivatives of \eqref{scatteringLambda},
\[\frac{\varphi_{\infty 1}'}{\varphi_{\infty 1}}(s,\chi) = -\log q - 2\frac{\Lambda'}{\Lambda}(2 - 2s,\chi) - 2\frac{\Lambda'}{\Lambda}(2s,\overline{\chi}).\]
Taking logarithmic derivatives of \eqref{Lambdadef} and letting $s = 1/2$ then shows that
\[\frac{\varphi_{\infty 1}'}{\varphi_{\infty 1}}\left(\frac{1}{2},\chi\right) = -4\Re\left(\frac{L'}{L}(1,\chi)\right) - 2 \log q + \log 8\pi + \gamma_0 + (-1)^{\kappa} \frac{\pi}{2},\]
where $\gamma_0$ denotes the Euler--Mascheroni constant, and we have used the fact that
\[\frac{\Gamma'}{\Gamma}\left(\frac{1 + \kappa}{2}\right) = - \log 8 - \gamma_0 - (-1)^{\kappa} \frac{\pi}{2}.\]
So if $\eta \ll 1/q$,
\[\II \ll \frac{\left(|L(1,\chi)| \log qT + |L'(1,\chi)|\right)}{q \eta |L(1,\chi)|}.\]
The desired upper bound then follows from the bounds
\[|L(1,\chi)| \ll \log q, \qquad |L'(1,\chi)| \ll (\log q)^2,\]
which are both easily shown via partial summation. See, for example, \cite[Lemma 10.15]{Montgomery} for the former estimate; the latter follows by a similar argument.
\end{proof}

\subsection{Lower Bounds for \texorpdfstring{$\II$}{I}}

\begin{proposition}
If $T \geq 1$ and $\eta = 1/T$, we have the lower bound
\[\II \gg \frac{1}{q |L(1,\chi)|^2} \sum_{T \leq m \leq 2T} \left|\sigma_0(m,\chi)\right|^2.\]
\end{proposition}

\begin{proof}
If $\eta = 1/T$, then \hyperref[ImzImgammaz]{Lemma \ref*{ImzImgammaz}} implies that
\[\Lambda^T E_{\infty}(z,s,\chi) = \begin{dcases*}
E_{\infty}(z,s,\chi) & if $1/T < \Im(z) \leq T$,	\\
E_{\infty}(z,s,\chi) - c_{\infty\infty}(z,s,\chi) & if $\Im(z) > T$.
\end{dcases*}\]
It follows that the nonzero Fourier coefficients of $\Lambda^T E_{\infty}(z,s,\chi)$ coincide with those of $E_{\infty}(z,s,\chi)$ for $\Im(z) > 1/T$. So by Parseval's identity, using the fact that $|\tau(\chi)| = \sqrt{q}$, and making the change of variables $y \mapsto y/|m|$ in the integral, we have that
\[\II \gg \begin{dcases*}
\frac{1}{q |L(1,\chi)|^2} \sum_{m = 1}^{\infty} \left|\sigma_0(m,\chi)\right|^2 \int_{m/T}^{\infty} \left|K_0(2\pi y)\right|^2 \, \frac{dy}{y} & if $\kappa = 0$,	\\
\frac{1}{q |L(1,\chi)|^2} \sum_{m = 1}^{\infty} \left|\sigma_0(m,\chi)\right|^2 \int_{m/T}^{\infty} e^{-4\pi y} \, \frac{dy}{y} & if $\kappa = 1$.
\end{dcases*}\]
If we simply consider the contribution of the positive integers $m$ for which $m/T \asymp 1$ --- say $T \leq m \leq 2T$ --- then we find that
\[\II \gg \frac{1}{q |L(1,\chi)|^2} \sum_{T \leq m \leq 2T} \left|\sigma_0(m,\chi)\right|^2,\]
as desired.
\end{proof}

Combining the upper and lower bounds for $\II$, we derive the following inequality for $L(1,\chi)$.

\begin{corollary}\label{L1chilowercor}
For all $T \geq q$, we have that
\[|L(1,\chi)| \gg \frac{1}{T (\log T)^2} \sum_{T \leq m \leq 2T} \left|\sigma_0(m,\chi)\right|^2.\]
\end{corollary}

So to obtain lower bounds for $|L(1,\chi)|$, we must find lower bounds for
\begin{equation}\label{sievebound}
\sum_{T \leq m \leq 2T} \left|\sigma_0(m,\chi)\right|^2.
\end{equation}

\subsection{Sieve Methods}

For quadratic characters, lower bounds for \eqref{sievebound} follow by restricting the sum to perfect squares.

\begin{lemma}\label{sigmaquad}
If $\chi$ is a quadratic character, then
\[\sum_{T \leq m \leq 2T} \left|\sigma_0(m,\chi)\right|^2 \geq (\sqrt{2} - 1) \sqrt{T}.\]
\end{lemma}

\begin{proof}
We restrict the sum over $m$ to perfect squares and use the fact that $\sigma_0(m,\chi) \geq 1$ whenever $m$ is a perfect square in order to find that
\[\sum_{T \leq m \leq 2T} \left|\sigma_0(m,\chi)\right|^2 \geq \sum_{T \leq m^2 \leq 2T} \left|\sigma_0(m^2,\chi)\right|^2 \geq (\sqrt{2} - 1) \sqrt{T}.\qedhere\]
\end{proof}

For complex characters, we instead restrict the sum in \eqref{sievebound} to primes and use the Brun--Titchmarsh inequality to show that there are sufficiently many primes for which $\overline{\chi}(p)$ is not close to $-1$, so that $|\sigma_0(p,\chi)|^2$ is not small. This is a result of Balasubramanian and Ramachandra \cite[Lemma 4]{Balasubramanian}, who combine it with an identity of Ramanujan together with a complex analytic argument to obtain lower bounds for $L(1 + it,\chi)$, and consequently derive zero-free regions for $L(s,\chi)$. We reproduce a proof of this result here for the sake of completeness.

\begin{lemma}[{Balasubramanian--Ramachandra \cite[Lemma 4]{Balasubramanian}}]\label{sigmacomplex}
There exists a large constant $K \geq 2$ such that for all complex characters $\chi$ modulo $q$ with $q \geq 2$ and for $T = q^K$,
\[\sum_{T \leq m \leq 2T} \left|\sigma_0(m,\chi)\right|^2 \gg_K \frac{T}{\log T}.\]
\end{lemma}

\begin{proof}
We restrict the sum over $m$ to primes $p$ in order to find that
\begin{align*}
\sum_{T \leq m \leq 2T} \left|\sigma_0(m,\chi)\right|^2 & \geq \sum_{T \leq p \leq 2T} \left|1 + \chi(p)\right|^2	\\
& = 2 \sum_{a \in (\Z/q\Z)^{\times}} (1 + \Re(\chi(a))) (\pi(2T;q,a) - \pi(T;q,a)),
\end{align*}
where
\[\pi(x;q,a) \defeq \#\left\{p \leq x \colon p \equiv a \hspace{-.2cm} \pmod{q}\right\}.\]
Let $Q$ be the order of the Dirichlet character $\chi$; this divides $\varphi(q)$, and as $\chi$ is complex, $Q \geq 3$. For any integer $M$ between $0$ and $\lfloor Q/2 \rfloor$, we have that
\begin{multline*}
\sum_{T \leq m \leq 2T} \left|\sigma_0(m,\chi)\right|^2 \geq 2 \left(1 + \cos \frac{2\pi M}{Q}\right) (\pi(2T) - \pi(T))	\\
- 2 \left(1 + \cos \frac{2\pi M}{Q}\right) \sum_{\substack{a \in (\Z/q\Z)^{\times} \\ \Re(\chi(a)) < \cos \frac{2\pi M}{Q}}} (\pi(2T;q,a) - \pi(T;q,a)).
\end{multline*}

For the former sum, we have that for fixed $\delta > 0$ to be chosen,
\[\pi(2T) - \pi(T) \geq (1 - \delta) \frac{T}{\log T}\]
for all sufficiently large $T$ dependent on $\delta$. See, for example, \cite{Diamond}; in particular, this does not require the full strength of the prime number theorem.

For the latter sum, we first observe that there are $\varphi(q)/Q$ reduced residue classes $a$ modulo $Q$ for which $\chi(a) = e^{2\pi i m/Q}$ for each integer $m$ between $0$ and $Q - 1$, and so the number of reduced residue classes modulo $q$ for which $\Re(\chi(a)) < \cos \frac{2\pi M}{Q}$ is
\[\frac{\varphi(q)}{Q} \#\{M < m < Q - M\} = \varphi(q)\frac{Q - 2M - 1}{Q}.\]
To find an upper bound for $\pi(2T;q,a) - \pi(T;q,a)$, we use the Brun--Titchmarsh inequality, which states that for $(q,a) = 1$, $x \geq 2$, and $y \geq 2q$,
\[\pi(x + y;q,a) - \pi(x;q,a) \leq \frac{2y}{\varphi(q) \log y/q}\left(1 + \frac{8}{\log y/q}\right).\]
We take $x = y = T$, assuming that $T \geq 2q$, in order to obtain
\[\sum_{\substack{a \in (\Z/q\Z)^{\times} \\ \Re(\chi(a)) < \cos \frac{2\pi M}{Q}}} (\pi(2T;q,a) - \pi(T;q,a)) \leq \frac{2(Q - 2M - 1)}{Q} \frac{T}{\log T/q} \left(1 + \frac{8}{\log T/q}\right).\]
We take $T = q^K$ with $K \geq 2$ sufficiently large and dependent on $\delta$ but not on $q$, such that
\[\frac{1}{\log T/q} \left(1 + \frac{8}{\log T/q}\right) \leq (1 + \delta) \frac{1}{\log T}.\]

Combined, these estimates imply that for $T = q^K$ with $K \geq 2$ a sufficiently large constant,
\[\sum_{T \leq m \leq 2T} \left|\sigma_0(m,\chi)\right|^2 \geq 2 (1 - \cos \pi X) \left(1 - \delta - 2(1 + \delta) X + \frac{2(1 + \delta)}{Q}\right) \frac{T}{\log T}\]
for
\[X = \frac{Q - 2M}{Q}.\]
For $Q \geq 3$, we may choose
\[\delta = \frac{1}{10}, \qquad M = \left\lfloor \frac{1 + 4\delta}{2(1 + \delta)} \frac{Q}{2} + \frac{1}{2} \right\rfloor,\]
so that
\[X = \frac{1 - 2\delta}{2(1 + \delta)} - \frac{1}{Q} + \frac{2}{Q}\left\{\frac{1 + 4\delta}{2(1 + \delta)} \frac{Q}{2} + \frac{1}{2}\right\},\]
and hence
\[1 - \delta - 2(1 + \delta) X + \frac{2(1 + \delta)}{Q} = \delta + \frac{4(1 + \delta)}{Q} \left(1 - \left\{\frac{1 + 4\delta}{2(1 + \delta)} \frac{Q}{2} + \frac{1}{2}\right\}\right) \geq \delta.\]
Moreover, the fact that $\delta = 1/10$ and $Q \geq 3$ implies that $1 \leq M \leq \lfloor Q/2 \rfloor$ and $1/33 \leq X \leq 23/33$. So
\[\sum_{T \leq m \leq 2T} \left|\sigma_0(m,\chi)\right|^2 \gg_K \frac{T}{\log T}.\qedhere\]
\end{proof}

\begin{remark}\label{LandauSiegelremark}
If $\chi$ is quadratic, so that the order of $\chi$ is $Q = 2$, then
\[\sum_{T \leq m \leq 2T} \left|\sigma_0(m,\chi)\right|^2 \geq 2 (\pi(2T) - \pi(T)) - 2 \sum_{\substack{a \in (\Z/q\Z)^{\times} \\ \chi(a) = -1}} (\pi(2T;q,a) - \pi(T;q,a)).\]
The Brun--Titchmarsh inequality is insufficient to show that the first term on the right-hand side dominates the second term; in its place, we would require a strengthening of the Brun--Titchmarsh inequality of the form
\begin{equation}\label{strongBrTeq}
\pi(x + y;q,a) - \pi(x;q,a) \leq \frac{(2 - \delta)y}{\varphi(q) \log y/q}\left(1 + o(1)\right)
\end{equation}
for some $\delta > 0$. With this in hand, we would then be able to show that
\[\sum_{T \leq m \leq 2T} \left|\sigma_0(m,\chi)\right|^2 \gg \frac{T}{\log T},\]
so that
\[L(1,\chi) \gg \frac{1}{(\log q)^3},\]
which would imply the nonexistence of a Landau--Siegel zero for $L(1,\chi)$. Of course, the fact that the strengthened Brun--Titchmarsh inequality \eqref{strongBrTeq} implies (and is in fact equivalent to) the nonexistence of Landau--Siegel zeroes is well-known.
\end{remark}

\section{Proof of \texorpdfstring{\hyperref[mainthm]{Theorem \ref*{mainthm}}}{Theorem \ref{mainthm}}}
 
With these upper and lower bounds established, we are in a position to prove \hyperref[mainthm]{Theorem \ref*{mainthm}}.

\begin{proof}[Proof of \texorpdfstring{\hyperref[mainthm]{Theorem \ref*{mainthm}}}{Theorem \ref{mainthm}}]
If $\chi$ is quadratic, we have via \hyperref[L1chilowercor]{Corollary \ref*{L1chilowercor}} and \hyperref[sigmaquad]{Lemma \ref*{sigmaquad}} that for $T \geq q$,
\[L(1,\chi) \gg \frac{1}{\sqrt{T}(\log T)^2},\]
and so taking $T = q$ yields the desired lower bound.

If $\chi$ is complex, we have via \hyperref[L1chilowercor]{Corollary \ref*{L1chilowercor}} and \hyperref[sigmacomplex]{Lemma \ref*{sigmacomplex}} that for $T = q^K$,
\[|L(1,\chi)| \gg_K \frac{1}{(\log T)^3} \gg_K \frac{1}{(\log q)^3},\]
as desired.
\end{proof}


\begin{thebibliography}{GeLa06}

\bibitem[Art80]{Arthur} James Arthur, ``A Trace Formula for Reductive Groups II: Applications of a Truncation Operator'', \textit{Compositio Mathematica} \textbf{40}:1 (1980), 87--121. \url{http://www.numdam.org/item?id=CM_1980__40_1_87_0}

\bibitem[BR76]{Balasubramanian} R.\ Balasubramanian and K.\ Ramachandra, ``The Place of an Identity of Ramanujan in Prime Number Theory'', \textit{Proceedings of the Indian Academy of Sciences, Section A} \textbf{83}:4 (1976), 156--165. \textsc{doi}:\allowbreak\href{http://dx.doi.org/10.1007/BF03051376}{10.1007/BF03051376}

\bibitem[DI82]{Deshouillers} J.-M.\ Deshouillers and H.\ Iwaniec, ``Kloosterman Sums and Fourier Coefficients of Cusp Forms'', \textit{Inventiones Mathematicae} \textbf{70}:2 (1982), 219--288. \textsc{doi}:\allowbreak\href{http://dx.doi.org/10.1007/BF01390728}{10.1007/BF01390728}

\bibitem[DE80]{Diamond} Harold G.\ Diamond and Paul Erd\H{o}s, ``On Sharp Elementary Prime Number Estimates'', \textit{L'Enseignement Math\'{e}matique} \textbf{26}:2 (1980), 313--321. \textsc{doi}:\allowbreak\href{http://dx.doi.org/10.5169/seals-51076}{10.5169/seals-51076}

\bibitem[DFI02]{Duke} W.\ Duke, J.\ B.\ Friedlander, and H.\ Iwaniec, ``The Subconvexity Problem for Artin $L$-Functions'', \textit{Inventiones Mathematicae} \textbf{149}:3 (2002), 489--577. \textsc{doi}:\allowbreak\href{http://dx.doi.org/10.1007/s002220200223}{10.1007/s002220200223}

\bibitem[GeLa06]{Gelbart} Stephen S.\ Gelbart and Erez M.\ Lapid, ``Lower Bounds for $L$-Functions at the Edge of the Critical Strip'', \textit{American Journal of Mathematics} \textbf{128}:3 (2006), 619--638. \textsc{doi}:\allowbreak\href{http://dx.doi.org/10.1353/ajm.2006.0024}{10.1353/ajm.2006.0024}

\bibitem[GoLi16]{Goldfeld} Dorian Goldfeld and Xiaoqing Li, ``A Standard Zero Free Region for Rankin Selberg $L$-Functions'', preprint (2016), 67 pages. arXiv:\allowbreak\href{https://arxiv.org/abs/1606.00330}{math.NT/1606.00330}

\bibitem[GR07]{Gradshteyn} I.\ S.\ Gradshteyn and I.\ M.\ Ryzhik, \textit{Table of Integrals, Series, and Products, Seventh Edition}, editors Alan Jeffrey and Daniel Zwillinger, Academic Press, Burlington, 2007.

\bibitem[Iwa02]{Iwaniec} Henryk Iwaniec, \textit{Spectral Methods of Automorphic Forms, Second Edition}, Graduate Studies in Mathematics \textbf{53}, American Mathematical Society, Providence, 2002.

\bibitem[Miy06]{Miyake} Toshitsune Miyake, \textit{Modular Forms}, Springer Monographs in Mathematics, Springer, Berlin, 2006. \textsc{doi}:\allowbreak\href{http://dx.doi.org/10.1007/3-540-29593-3}{10.1007/3-540-29593-3}

\bibitem[MV07]{Montgomery} Hugh L.\ Montgomery and Robert C.\ Vaughan, \textit{Multiplicative Number Theory I. Classical Theory}, Cambridge Studies in Advanced Mathematics \textbf{97}, Cambridge University Press, Cambridge, 2007.

\bibitem[Sar04]{Sarnak} Peter Sarnak, ``Nonvanishing of $L$-Functions on $\Re(s) = 1$'', in \textit{Contributions to Automorphic Forms, Geometry \& Number Theory}, editors Haruzo Hida, Dinakar Ramakrishnan, and Freydoon Shahidi, The John Hopkins University Press, Baltimore, 2004, 719--732.

\end{thebibliography}
\end{document}